\documentclass[11pt]{amsart}

\usepackage{amssymb,amsmath,amsfonts}
\usepackage{graphicx}

\newtheorem{theorem}{Theorem}[section] % 1st argument is your name for it
\newtheorem{lemma}[theorem]{Lemma}     % 2nd argument is what is printed
\newtheorem{corollary}[theorem]{Corollary}
\newtheorem{prop}[theorem]{Proposition}
\newtheorem{remark}[theorem]{Remark}
\def\neweq#1{\begin{equation}\label{#1}}
\def\endeq{\end{equation}}

\def\ri{\rightarrow}
\def\ep{\varepsilon}
\def\vphi{\varphi}
\def\la{\lambda}

\def\eq#1{(\ref{#1})}
\def\RR{{\mathbb R} }

\def\o{\Omega }
\def\oo{\overline\Omega }

\title[Singular elliptic inequalities in the exterior of a compact set]{Singular semilinear elliptic inequalities in the exterior of a compact set}
\author[M. Ghergu]{Marius Ghergu}

\address{School of Mathematical Sciences, University College Dublin, Belfield, Dublin 4, Ireland}

\email{marius.ghergu@ucd.ie}

\author[S.D. Taliaferro]{Steven D. Taliaferro}
\address{Department of Mathematics, Texas A$\&$M University,
College Station, TX 77843-3368, USA}

\email{stalia@math.tamu.edu}
\begin{document}
\maketitle

\begin{abstract}
We study the semilinear elliptic inequality $-\Delta
u\geq\vphi(\delta_K(x))f(u)$ in $\RR^N\setminus K,$ where $\vphi,
f$ are positive and nonincreasing  continuous functions. Here
$K\subset \RR^N$ $(N\geq 3)$ is a compact set with finitely many
components each of which is either the closure of a $C^2$ domain
or an isolated point and $\delta_K(x)={\rm dist}(x,\partial K)$.
We obtain optimal conditions in terms of $\vphi$ and $f$ for the
existence of  $C^2$ positive solutions. Under these conditions we
prove the existence of a minimal solution and we investigate its
behavior around $\partial K$ as well as the removability of the
(possible) isolated singularities.
\medskip


\end{abstract}

%\part{Use this type of header for very long papers only}
% use lowercase except for proper names

\section{Introduction} %

In this paper we study the existence and non-existence of $C^2$
positive solutions $u(x)$ of the following semilinear elliptic
inequality
\begin{equation}\label{phi}
-\Delta u\geq \vphi(\delta_K(x))f(u) \quad\mbox{ in
}\RR^N\setminus K,
\end{equation}
where $K$ is a compact set in $\RR^N$ ($N\geq 3$) and
$\delta_K(x):=$dist$(x,\partial K)$. We assume that $K$ has
finitely many connected components each of which is either the
closure of a $C^2$ domain or a singleton. We shall write
$K=K_1\cup K_2$ where $K_1$ is the union of all components of $K$
which are the closure of a $C^2$ domain and $K_2$ is the set of
all isolated points of $K$.

We also assume that

\medskip

\noindent$(A1)\qquad$ $f\in C^1(0,\infty)$ is a positive and
decreasing function;

\medskip

\noindent$(A2)\qquad$ $\vphi \in C^{0,\gamma}(0,\infty)$
($0<\gamma<1$) is a positive and nonincreasing function.

\medskip

Elliptic equations or inequalities in unbounded domains have been
subject to extensive study recently (see, e.g.,
\cite{dav1,dav2,gr,kon1,kon2,lisk,pucci1,pucci2} and the
references therein). In \cite{dav1,dav2} the authors are concerned with existence of solutions to $-\Delta u=u^p$ in $\RR^N\setminus\Omega$, where $\Omega\subset\RR^N$ is smooth and bounded. If $p>(N+2)/(N-2)$ it is obtained in \cite{dav1,dav2} that there exist infinitely many solutions that vanish on $\partial\Omega$ and decay slower than $|x|^{-2/(p-1)}$ at infinity.

In \cite{pucci1,pucci2} general elliptic inequalities of the form
$\pm {\rm div}\{A(|\nabla u|)\nabla u\} \geq f(u)$ in possibly unbounded domains are considered. Under some growing conditions on $A$ and $f$, the authors obtain existence of a solution. Large classes of elliptic inequalities in
exterior or cone-like domains involving various types of
differential operators  are considered in
\cite{kon1,kon2,lisk}. In
\cite{tal99,talc01,tal01,tal02,tal06} elliptic inequalities of the form $af(u)\leq -\Delta u\leq b f(u)$, $(b>a\geq 0)$ are
studied in a punctured neighborhood of the origin and asymptotic
radial symmetry of solutions is investigated. 
For $f$ and $\vphi$ that satisfy $(A1)$-$(A2)$, the equation $-\Delta u=\vphi(|x|)f(u)$ in $\RR^N$ was studied in \cite{lair1,lair2}. It is obtained that a solution exists if and only if $\int_0^\infty r\vphi(r)dr<\infty$.

The main novelty of the present paper is the presence of the
distance function $\delta_K(x)$ to the boundary of the compact set
$K$ which, as we shall see, will play a significant role in the
qualitative study of \eq{phi}. Whenever \eq{phi} has solutions we
show that it has a minimal solution $\tilde u$ and we are
interested in further properties of $\tilde u$ such  as
removability of possible singularities at isolated points of $K_2$
as well as boundary behavior around $K_1$.

In our approach to \eq{phi} we shall distinguish between the case
where $K$ is {\it non-degenerate}, that is, $K_1\neq \emptyset$,
and the case where $K$ is {\it degenerate}, that is
$K_1=\emptyset$, which means $K$ reduces to a finite set of
points.

We start first with the non-degenerate case $K_1\neq\emptyset$.
Our first result in this sense is the following:

\begin{theorem}\label{thoptim2}
Assume $(A1)$, $(A2)$ and $K_1\neq \emptyset$. Then, inequality
\eq{phi} has $C^2$ positive solutions if and only if
\begin{equation}\label{vphioptim}
\int_0^\infty r\vphi(r) dr<\infty.
\end{equation}
\end{theorem}

If \eq{vphioptim} holds, then we prove that \eq{phi} has a minimal
$C^2$ positive solution $\tilde u$ (in the sense of the usual
order relation) which achieves the equality in \eq{phi} and
$\tilde u$ is a {\it ground-state} of \eq{phi} in the sense that
$\tilde u(x)\ri 0$ as $|x|\ri\infty$. Furthermore, we prove that
all (possible) singularities of $\tilde u$ at isolated points in
$K_2$ are removable and that $\tilde u$ can be continuously
extended on $\partial K_1$. We also determine the rate at
which $\tilde u$ vanishes around the boundary of $K_1$. All these
results are precisely described in the following theorem.

\begin{theorem}\label{thoptim3}
Assume  $(A1)$, $(A2)$, $K_1\neq \emptyset$ and condition
\eq{vphioptim} is satisfied. Then there exists a minimal solution
$\tilde u$ of \eq{phi} that satisfies
$$
\tilde u\in C^2(\RR^N\setminus K)\cap C(\RR^N\setminus int(K_1))
$$
and
\begin{equation}\label{edemc}
\left\{\begin{aligned}
-&\Delta \tilde u=\vphi(\delta_K(x))f(\tilde u), \; \tilde u>0&& \quad\mbox{ in }\RR^N\setminus K,\\
&\tilde u=0 &&\quad\mbox{ on }\partial K_1,\\
&\tilde u>0 &&\quad\mbox{ on }K_2,\\
&\tilde u(x)\ri 0 &&\quad\mbox{ as }|x|\ri\infty.
\end{aligned}\right.
\end{equation}
In addition, there exist positive constants $c_1,c_2,$ and $r_0$
such that $\tilde u$ satisfies
\begin{equation}\label{bestim}
c_1\leq \frac{\tilde u(x)}{H(\delta_{K_1}(x))}\leq c_2 \quad\mbox{
in }\{x\in\RR^N\setminus K:0<\delta_{K_1}(x)<r_0\},
\end{equation}
where  $H:[0,1]\ri [0,\infty)$ is the unique solution of
\begin{equation}\label{oregan}
\left\{\begin{aligned}
-&H''(t)=\vphi(t)f(H(t)), \; H(t)>0&& \quad0<t<1,\\
&H(0)=H(1)=0. &&
\end{aligned}\right.
\end{equation}
\end{theorem}
The existence of a solution to \eq{oregan} follows from
\cite[Theorem 2.1]{aga}. By Theorem \ref{thoptim2} we have that
condition \eq{vphioptim} is both necessary and sufficient for the
existence of a solution to \eq{phi}. If that is the case, the
minimal solution $\tilde u$ of \eq{phi} can be continuously
extended to $\partial K$, so that all isolated singularities of
$\tilde u$ at $K_2$ are removable. If $\vphi(r)=r^{\alpha}$ and
$f(u)=u^{-p}$, $p>0$, the behavior of $H$ in \eq{oregan} was
studied in  \cite[Theorem 3.5]{gr2}. In this case we have:

\begin{corollary}\label{cor1}
Assume $(A2)$, $K_1\neq \emptyset$, $f(u)=u^{-p},p>0,$ and
$$
\vphi(r)\sim r^{\alpha}\quad\mbox{ as }r\ri 0\quad \mbox{ and
}\quad \vphi(r)\sim r^{\beta}\quad\mbox{ as }r\ri \infty\,,
$$
for some $\alpha, \beta<0$. Then \eq{phi} has solutions if and
only if $0>\alpha>-2>\beta$. In this case \eq{phi} has a minimal
solution $\tilde u$ which satisfies \eq{edemc} and there exist
positive constants $c_1,c_2,$ and $r_0$ such that $\tilde u$
satisfies \eq{bestim} where
$$H(t)=
\left\{\begin{aligned}
&t&&\quad\mbox{ if } p-\alpha<1,\\
&t\left(\log\frac{1}{t}\right)^{\frac{1}{2+\alpha}}&&\quad\mbox{ if } p-\alpha=1,\\
&t^{\frac{2+\alpha}{1+p}}&&\quad\mbox{ if } p-\alpha>1.
\end{aligned}\right.
$$
\end{corollary}

We are next concerned with the degenerate case $K_1=\emptyset$. In
this setting the existence of a solution to \eq{phi} depends on
both $\vphi$ and $f$. Our result in this case is:

\begin{theorem}\label{k2}
Assume $(A1)$, $K_1=\emptyset$ and that $\vphi \in
C^{0,\gamma}(0,\infty)$ ($0<\gamma<1$) is a positive function
which is nonincreasing in a neighborhood of zero and of infinity.
Then, \eq{phi} has solutions if and only if
\begin{equation}\label{eqp1}
\int_1^\infty r\vphi(r) dr<\infty
\end{equation}
and there exists $a>0$ such that
\begin{equation}\label{eqp2}
\int_0^1r^{N-1}\vphi(r) f(ar^{2-N}) dr<\infty.
\end{equation}
Furthermore, if \eq{eqp1}-\eq{eqp2} hold, then \eq{phi} has a
minimal solution $\tilde u$ which satisfies
\begin{equation}\label{mins}
\left\{\begin{aligned}
-&\Delta \tilde u=\vphi(\delta_K(x))f(\tilde u), \; \tilde u>0&& \quad\mbox{ in }\RR^N\setminus K,\\
&\tilde u(x)\ri 0 &&\quad\mbox{ as }|x|\ri\infty.
\end{aligned}\right.
\end{equation}
In addition, $\tilde u$ has removable singularities at $K$ if and
only if $\int_0^1 r\vphi(r)dr<\infty$.
\end{theorem}
\noindent From Theorem \ref{thoptim2} and Theorem \ref{k2} we have
the following result regarding the inequality
\begin{equation}\label{phi1}
-\Delta u\geq \delta^{\alpha}_K(x) u^{-p} \quad\mbox{ in
}\RR^N\setminus K,\quad \alpha<0<p.
\end{equation}

\begin{corollary}
Assume $K$ has finitely many components and $K=K_1\cup K_2$ where 
$K_1$ is the union of all components of $K$
which are the closure of a $C^2$ domain and $K_2$ is the set of
all isolated points of $K$.
\begin{itemize}
\item[(i)] If $K_1$ is nonempty, then \eq{phi1} has no positive
$C^2$ solutions; \item[(ii)] If $K_1=\emptyset$ then \eq{phi1} has
solutions if and only if
\begin{equation}\label{aabbb}
N+\alpha+p(N-2)>0\quad\mbox{ and }\quad \alpha<-2,
\end{equation}
and all solutions of \eq{phi1} are singular at points of $K_2$.
\end{itemize}
\end{corollary}

\noindent Finally, we consider the special case $K=\{0\}$ and
describe the solution set of
\begin{equation}\label{rad}
-\Delta u= \vphi(|x|)f(u)  \quad\mbox{ in }\RR^N\setminus\{0\}.
\end{equation}

For a large class of functions $\vphi$, we show that any $C^2$
positive solution of \eq{rad} (if exists) is radially symmetric.
Furthermore, the solution set of \eq{rad} consists of a
two-parameter family of radially symmetric functions.

\begin{theorem}\label{thdeg2} Suppose that $f$ and $\vphi$ are as in Theorem \ref{k2} and that $\vphi$
satisfies \eq{eqp1}-\eq{eqp2} for all $a>0$. Then :
\begin{enumerate}
\item[(i)] for any $a,b\geq 0$ there exists a radially symmetric
positive solution $u_{a,b}$ of \eq{rad} such that
\begin{equation}\label{lim}
\lim_{|x|\ri 0}|x|^{N-2}u_{a,b}(x)=a\quad\mbox{ and }\quad
\lim_{|x|\ri \infty}u_{a,b}(x)=b.
\end{equation}
\item[(ii)] the set of positive solutions of equation \eq{rad}
consists only of $\{u_{a,b}: \, a,b\geq 0\}$. In particular, any
$C^2$ positive solution of \eq{rad} is radially symmetric.
\end{enumerate}
\end{theorem}

We point out that if $N=2$ then \eq{rad} has no $C^2$ positive
solutions.  More precisely, if $u\in C^2(\RR^2\setminus\{0\})$
satisfies $-\Delta u\geq 0$, $u\geq 0$ in $\RR^2\setminus\{0\}$,
then $u$ is constant (see \cite[Theorem 29, page 130]{protter}). A
direct consequence of Theorem \ref{thdeg2} is the following:

\begin{corollary}\label{cor2} Let $\alpha\in\RR$, $p>0$. Then, the equation
\begin{equation}\label{e1}
-\Delta u= |x|^\alpha u^{-p} \quad\mbox{ in
}\RR^N\setminus\{0\},\, N\geq 3,
\end{equation}
has positive solutions if and only if \eq{aabbb} holds. In this
case, we have the same conclusion as in Theorem \ref{thdeg2} and
the function
\begin{equation}\label{xii}
\xi(x):=\left[\frac{-(1+p)^2}{(\alpha+2)( p(N-2)+N+\alpha)}
\right]^{1/(1+p)}|x|^{(2+\alpha)/(1+p)},\quad
x\in\RR^N\setminus\{0\},
\end{equation}
is the minimal solution of \eq{e1}.
\end{corollary}

\smallskip

Using Theorem \ref{thdeg2} we also obtain:

\smallskip

\begin{corollary}\label{c2} Let $\alpha\in\RR$, $\beta,p>0$.
Then, the equation
\begin{equation}\label{e2}
-\Delta u= |x|^\alpha\log^{\beta}(1+|x|) u^{-p} \quad\mbox{ in
}\RR^N\setminus\{0\},\, N\geq 3,
\end{equation}
has solutions if and only if
\begin{equation}\label{aabbbb}
N+\alpha+\beta+p(N-2)>0\quad\mbox{ and }\quad \alpha<-2.
\end{equation}
Furthermore, if \eq{aabbbb} holds, then:
\begin{enumerate}
\item[(i)] the set of positive solutions of \eq{e2} consists of a
two-parameter family of radially symmetric functions as described
in Theorem \ref{thdeg2}; \item[(ii)] the minimal solution of
\eq{e2} has a removable singularity at the origin if and only if
$\alpha+\beta>-2$.
\end{enumerate}
\end{corollary}

\smallskip

The outline of the paper is as follows. In the next section we
collect some preliminary results concerning elliptic boundary
value problems in bounded domains involving the distance function
up to the boundary. The last four sections of the paper are
devoted to the proofs of Theorems \ref{thoptim2}, \ref{thoptim3},
\ref{k2} and \ref{thdeg2} respectively.

\section{Preliminary results}

In this part we obtain some results for related elliptic problems
in bounded domains that will be further used in the sequel. We
start with the following comparison result.

\begin{lemma}\label{comp}Let $\Omega\subset\RR^N$ $(N\geq 2)$ be a nonempty open set
and $g:\Omega\times (0,\infty)\rightarrow (0,\infty)$ be a
continuous function such that $g(x,\cdot)$ is decreasing for all
$x\in\Omega$. Assume that $u,v$ are $C^2$ positive functions that
satisfy
$$\Delta u+g(x,u)\leq 0\leq \Delta v+g(x,v)\quad\mbox{ in }\Omega,$$
$$\lim_{x\in\Omega, \,x\rightarrow y}(v(x)-u(x))\leq 0\quad\mbox{ for all }y\in\partial^\infty\Omega.$$
Then $u\geq v$ in $\Omega$. (Here $\partial^\infty\Omega$ stands
for the Euclidean boundary $\partial\Omega$ if $\Omega$ is bounded
and for $\partial\Omega\cup\{\infty\}$ if $\Omega$ is unbounded)
\end{lemma}
\begin{proof} Assume by contradiction that the set
$\omega:=\{x\in\Omega: u(x)<v(x)  \}$ is not empty and let
$w:=v-u$. Since $\lim_{x\in\Omega,\, x\rightarrow y}w(x)\leq 0$
for all $y\in\partial^\infty\Omega$, it follows that $w$ is
bounded from above and it achieves its maximum on $\Omega$ at a
point that belongs to $\omega$. At that point, say $x_0$, we have
$$
0\leq -\Delta w(x_0)\leq g(x_0,v(x_0))-g(x_0,u(x_0))<0,
$$ which is a
contradiction. Therefore, $\omega=\emptyset$, that is, $u\geq v$
in $\Omega$.
\end{proof}

\begin{lemma}\label{crt} Let $\Omega\subset\RR^N$ $(N\geq 2)$ be a bounded domain with $C^2$ boundary and let
$g:\overline \Omega\times (0,\infty)\rightarrow (0,\infty)$ be a
H\"older continuous function such that for all
$x\in\overline\Omega$ we have $g(x,\cdot)\in C^1(0,\infty)$ and
$g(x,\cdot)$ is decreasing. Then, for any $\phi\in
C(\partial\Omega)$, $\phi\geq 0$, the problem
\begin{equation}\label{phii}
\left\{\begin{aligned}
-&\Delta u=g(x,u), \; u>0&& \quad\mbox{ in }\Omega,\\
&u=\phi(x) &&\quad\mbox{ on }\partial\Omega,
\end{aligned}\right.
\end{equation}
has a unique solution $u\in C^2(\Omega)\cap C(\overline\Omega)$.
\end{lemma}
\begin{proof} For all $n\geq 1$ consider the following
perturbed problem
\begin{equation}\label{phiip}
\left\{\begin{aligned}
-&\Delta u=g\Big(x,u+\frac{1}{n} \Big), \; u>0&& \quad\mbox{ in }\Omega,\\
&u=\phi(x) &&\quad\mbox{ on }\partial\Omega.
\end{aligned}\right.
\end{equation}
It is easy to see that $\underline u\equiv 0$ is a sub-solution.
To construct a super-solution, let $w$ be the solution of
$$
\left\{\begin{aligned}
-&\Delta w=1, \; w>0&& \quad\mbox{ in }\Omega,\\
&w=0 &&\quad\mbox{ on }\partial\Omega.
\end{aligned}\right.
$$
Then $\overline u=Mw+||\phi||_\infty+1$ is a super-solution of
\eq{phiip} provided $M>1$ is large enough. Thus, by sub and
super-solution method and Lemma \ref{comp}, there exists a unique
solution $u_n \in C^2(\Omega) \cap C(\overline \Omega)$ of
\eqref{phiip}. Furthermore, since $g(x,\cdot)$ is decreasing, by
Lemma \ref{comp} we deduce
\begin{equation}\label{lcomp1}
u_1\leq u_2\leq \dots \leq u_n\leq\dots \leq \overline u
\quad\mbox{ in }\Omega,
\end{equation}
\begin{equation}\label{lcomp2}
u_n+\frac{1}{n}\geq u_{n+1}+\frac{1}{n+1} \quad\mbox{ in }\Omega.
\end{equation}
Hence $\{u_n(x)\}$ is increasing and bounded for all $x\in\Omega$.
Letting $u(x):=\lim_{n\rightarrow \infty}u_n(x)$, a standard
bootstrap argument (see \cite{crrt}, \cite{gt}) implies
$u_n\rightarrow u$ in $C^2_{loc}(\Omega)$ so that passing to the
limit in \eqref{phiip} we deduce $-\Delta u=g(x,u)$ in $\Omega$.
From \eqref{lcomp1} and \eqref{lcomp2} we obtain $u_n+1/n\geq
u\geq u_n$ in $\Omega$, for all $n\geq 1$. This yields $u\in
C(\overline\Omega)$ and $u=\phi(x)$ on $\partial\Omega$. Therefore
$u\in C^2(\Omega)\cap C(\overline\Omega)$ is a solution of
\eqref{phii}. The uniqueness follows from Lemma \ref{comp}.
\end{proof}

Lemma \ref{comp2} and Lemma \ref{comp3} below extend the existence
results obtained in \cite{gr1,gr2,gr3}.

\begin{lemma}\label{comp2} Let $\o\subset\RR^N$ $(N\geq 2)$ be a bounded domain with $C^2$
boundary. Also let $\vphi \in C^{0,\gamma}(0,\infty)$
$(0<\gamma<1)$ and $f\in C^1(0,\infty)$ be positive functions such
that:
\begin{enumerate}
\item[(i)] $f$ is decreasing; \item[(ii)] $\vphi$ is nonincreasing
and $\int_0^1 r\vphi(r)dr<\infty$.
\end{enumerate}
Then, the problem
\begin{equation}\label{ede}
\left\{\begin{aligned}
-&\Delta u=\vphi(\delta_\o(x))f(u), \; u>0&& \quad\mbox{ in }\Omega,\\
&u=0 &&\quad\mbox{ on }\partial\Omega,
\end{aligned}\right.
\end{equation}
has a unique solution $u\in C^2(\o)\cap C(\oo)$. Furthermore,
there exist $c_1,c_2>0$ and $0<r_0<1$ such that the unique
solution $u$ of \eq{ede} satisfies
\begin{equation}\label{best}
c_1\leq \frac{u(x)}{H(\delta_\o(x))}\leq c_2 \quad\mbox{ in
}\{x\in\o: 0<\delta_\o(x)<r_0\},
\end{equation}
where  $H:[0,1]\ri (0,\infty)$ is the unique solution of
\eq{oregan}.
\end{lemma}
\begin{proof}  Let $(\la_1,e_1)$ be the first eigenvalue
and the first eigenfunction of $-\Delta$ in $\o$ subject to
Dirichlet boundary condition. It is well known that $e_1$ has
constant sign in $\o$ so that normalizing, we may assume that
$e_1>0$ in $\o$. Also, since $\o$ has a $C^2$ boundary, we have
$\partial e_1/\partial\nu<0$ on $\partial\o$ and
\begin{equation}\label{eigenf}
C_1\delta_\o(x)\leq e_1(x)\leq C_2\delta_\o(x)\quad\mbox{ in }\o,
\end{equation}
where $\nu$ is the outward unit normal at $\partial\o$ and
$C_1,C_2$ are two positive constants. We claim that there exist
$M>1$ and $c>0$ such that $\overline u=M H(c e_1)$ is a
super-solution of \eq{ede}. First, since the solution $H$ of
\eq{oregan} is positive and concave, we can find $0<a<1$ such that
$H'>0$ on $(0,a]$. Let $c>0$ be such that
\[
ce_1(x)\leq \min\{a,\delta_\o(x)\}\quad\mbox{ in }\o.
\]
Then
\begin{equation}\label{dell}
\begin{aligned}
-\Delta \overline u&=-Mc^2H''(ce_1)|\nabla
e_1|^2+Mc\la_1e_1H'(ce_1)\\
&=Mc^2\vphi(ce_1) f(H(ce_1))|\nabla
e_1|^2+Mc\la_1e_1H'(ce_1)\\
&\geq Mc^2\vphi(\delta_\o(x)) f(\overline u)|\nabla
e_1|^2+Mc\la_1e_1H'(ce_1)\quad\mbox{ in }\o.
\end{aligned}
\end{equation}
Since $e_1>0$ in $\o$ and $\partial e_1/\partial\nu<0$ on
$\partial\o$, we can find $d>0$ and a subdomain
$\omega\subset\subset\o$ such that
\[
|\nabla e_1|>d\quad\mbox{ in } \o\setminus\omega.
\]
Therefore, from \eq{dell} we obtain
\begin{equation}\label{separate}
-\Delta \overline u\geq Mc^2d^2 \vphi(\delta_\o(x)) f(\overline
u)\quad\mbox{ in }\o\setminus\omega, \qquad-\Delta \overline u\geq
Mc\la_1e_1H'(ce_1)\quad\mbox{ in }\omega.
\end{equation}
Now, we choose $M>0$ large enough such that
\begin{equation}\label{separatee}
Mc^2d^2>1\quad\mbox{ and}\quad Mc\la_1e_1H'(ce_1)\geq
\vphi(\delta_\o(x)) f(\overline u) \quad\mbox{ in }\omega.
\end{equation}
Note that the last relation in \eq{separatee} is possible since in
$\omega$ the right side of the inequality is bounded and the left
side is bounded away from zero.
 Thus, from \eq{separate} and \eq{separatee}, $\overline
u$ is a super-solution for \eq{ede}. Similarly, we can choose
$m>0$ small enough such that $\underline u=mH(ce_1)$ is a
sub-solution of \eq{ede}. Therefore, by the sub and super-solution
method we find a solution $u\in C^2(\o)\cap C(\oo)$ such that
$\underline u\leq u\leq \overline u$ in $\o$. The uniqueness
follows from Lemma \ref{comp}. In order to prove the boundary
estimate \eq{best}, note first that $ce_1\leq \delta_\o(x)$ in
$\o$ so
$$
u(x)\leq \overline u(x)\leq MH(\delta_\o(x))\quad\mbox{ in
}\{x\in\o: 0<\delta_\o(x)<a\}.
$$
On the other hand, since $H$ is concave and $H(0)=0$, we easily
derive that $t\ri H(t)/t$ is decreasing on $(0,1)$. Also we can
assume $cC_1<1$. Thus,
\[
u(x)\geq mH(ce_1)\geq m H(cC_1\delta_\o(x))\geq mcC_1
H(\delta_\o(x)),
\]
for all $x\in\o$ with $0<\delta_\o(x)<1$. The proof of Lemma
\ref{comp2} is now complete.
\end{proof}

\begin{lemma}\label{comp3} Let $K\subset\RR^N$ $(N\geq 2)$ be a compact set, $\o\subset\RR^N$ be a
bounded domain such that $K\subset \o$ and $\o\setminus K$ is
connected and has a $C^2$ boundary. Let $\vphi$ and $f$ be as in
Lemma \ref{comp2}. Then, there exists a unique solution $u\in
C^2(\o\setminus K)\cap C(\oo\setminus int(K))$ of the problem
\begin{equation}\label{edec}
\left\{\begin{aligned}
-&\Delta u=\vphi(\delta_K(x))f(u), \; u>0&& \quad\mbox{ in }\Omega\setminus K,\\
&u=0 &&\quad\mbox{ on }\partial(\Omega\setminus K).
\end{aligned}\right.
\end{equation}
Furthermore, there exist $c_1,c_2>0$ and $0<r_0<1$ such that the
unique solution $u$ of \eq{edec} satisfies
\begin{equation}\label{besti}
c_1\leq \frac{u(x)}{H(\delta_K(x))}\leq c_2 \quad\mbox{ in
}\{x\in\o\setminus K:0<\delta_K(x)<r_0\},
\end{equation}
where  $H$ is the unique solution of \eq{oregan}.
\end{lemma}
\begin{proof} According to Lemma \ref{comp2} there exists
$v\in C^2(\o\setminus K)\cap C(\overline{\o\setminus K})$ such
that
$$
\left\{\begin{aligned}
-&\Delta v=\vphi(\delta_{\o\setminus K}(x))f(v), \; v>0&& \quad\mbox{ in }\Omega\setminus K,\\
&v=0 &&\quad\mbox{ on }\partial(\Omega\setminus K),
\end{aligned}\right.
$$
which further satisfies
\begin{equation}\label{bestc}
c_1\leq \frac{v(x)}{H(\delta_{\o\setminus K}(x))}\leq c_2
\quad\mbox{ in }\{x\in\o\setminus K:0<\delta_{\o\setminus
K}(x)<\rho_0\},
\end{equation}
for some $0<\rho_0<1$ and $c_1,c_2>0$. Since $\delta_K(x)\geq
\delta_{\o\setminus K}(x)$ for all $x\in \Omega\setminus K$ and
$\vphi$ is nonincreasing, it is easy to see that $\overline u=v$
is a super-solution of \eq{edec}. Also it is not difficult to see
that $\underline u=m w$ is a sub-solution to \eq{edec} for $m>0$
sufficiently small, where $w$ satisfies
$$
\left\{\begin{aligned}
-&\Delta w=1, \; w>0&& \quad\mbox{ in }\Omega\setminus K,\\
&w=0 &&\quad\mbox{ on }\partial(\Omega\setminus K).
\end{aligned}\right.
$$
Using Lemma \ref{comp} we have $\underline u\leq \overline u$ in
$\o\setminus K$. Therefore, there exists a solution $u\in
C^2(\o\setminus K)\cap C(\oo\setminus int(K))$ of \eq{edec}. As
before, the uniqueness follows from Lemma \ref{comp}. In order to
prove \eq{besti}, let $0<r_0<\rho_0$ be small such that
$$
\omega:=\{x\in\o\setminus
K:0<\delta_K(x)<r_0\}\subset\subset\Omega \quad\mbox{ and }\quad
\delta_{\o\setminus K}(x)=\delta_K(x)\quad\mbox{for all }x\in
\omega.
$$
Then, from \eq{bestc} we have
\[
u\leq \overline u\leq c_2H(\delta_K(x))\quad\mbox{ in } \omega.
\]
For the remaining part of \eq{besti}, let $M>1$ be such that
$Mu\geq v$ on $\partial\omega\setminus\partial K$. Also
$$
-\Delta (Mu)= M\vphi(\delta_K(x))f(u)\geq \vphi(\delta_K(x))
f(Mu)\quad \mbox{ in }\omega.
$$
By Lemma \ref{comp} we have $Mu\geq v$ in $\omega$ and from
\eq{bestc} we obtain the first inequality in \eq{besti}. This
concludes the proof.
\end{proof}

\medskip

The following result is a direct consequence of Lemma \ref{comp3}.

\begin{lemma}\label{comp4} Let $K_1, K_2, L\subset\RR^N$ $(N\geq 2)$ be three compact sets (see Figure 1) such that
\[
K_1\cap L=\emptyset, \quad K_2\subset int(L), \quad K_1, L \mbox{
are the closure of $C^2$ domains}.
\]
\begin{figure}[tbph]\label{convv2}
\begin{center}
\includegraphics[height=4cm]{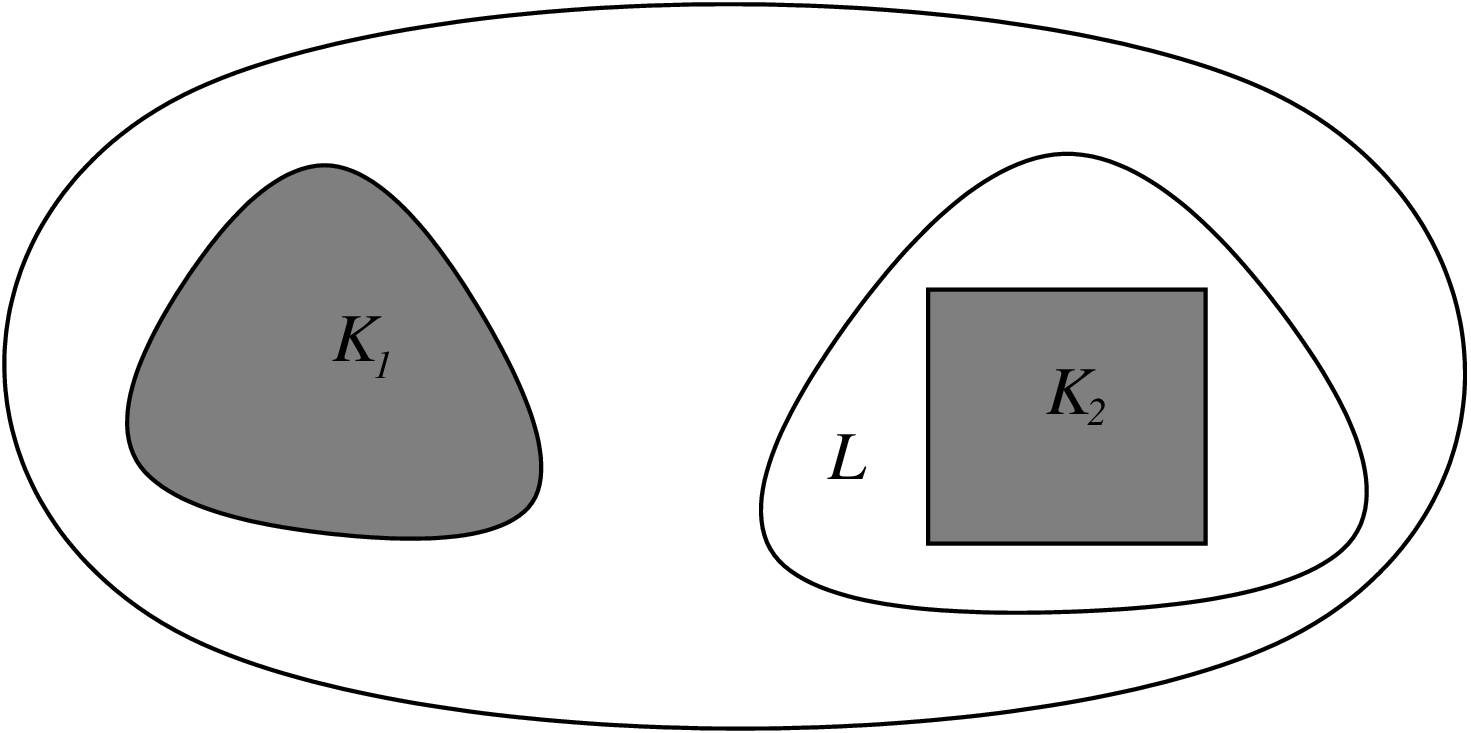}
\caption{The compact sets $K_1$, $K_2$ and $L$.}
\end{center}
\end{figure}

Also let $\o\subset\RR^N$ be a bounded domain with a $C^2$ boundary
such that $K_1\cup L\subset \o$ and $\o\setminus(K_1\cup L)$ is
connected. Let $\vphi, f$ be as in Lemma \ref{comp2}. Then, there
exists a unique solution
$$u\in C^2(\o\setminus(K_1\cup L))\cap C(\oo\setminus int(K_1\cup
L))$$ of the problem
\begin{equation}\label{edecu}
\left\{\begin{aligned}
-&\Delta u=\vphi(\delta_{K_1\cup K_2}(x))f(u), \; u>0&& \quad\mbox{ in }\Omega\setminus (K_1\cup L),\\
&u=0 &&\quad\mbox{ on }\partial(\Omega\setminus (K_1\cup L)).
\end{aligned}\right.
\end{equation}
Furthermore, there exist $c_1,c_2>0$ and $0<r_0<1$ such that the
unique solution $u$ of problem \eq{edecu} satisfies
\begin{equation}\label{bestiu}
c_1\leq \frac{u(x)}{H(\delta_{K_1}(x))}\leq c_2 \quad\mbox{ in
}\{x\in\o\setminus (K_1\cup L):0<\delta_{K_1}(x)<r_0\},
\end{equation}
where  $H$ is the unique solution of \eq{oregan}.
\end{lemma}

\begin{proof} By Lemma \ref{comp3} there exists a unique
$v\in C^2(\o\setminus(K_1\cup L))\cap C(\oo\setminus int(K_1\cup
L))$ such that
$$
\left\{\begin{aligned}
-&\Delta v=\vphi(\delta_{K_1\cup L} (x))f(v), \; v>0&& \quad\mbox{ in }\Omega\setminus (K_1\cup L),\\
&v=0 &&\quad\mbox{ on }\partial(\Omega\setminus (K_1\cup L)).
\end{aligned}\right.
$$
Since $\delta_{K_1\cup L}(x)\leq \delta_{K_1\cup K_2}(x)$ in
$\o\setminus(K_1\cup L)$ and $\vphi$ is nonincreasing, we derive
that $\overline u=v$ is a super-solution of \eq{edecu}. As a
sub-solution we use $\underline u=mw$ where $m$ is sufficiently
small and $w$ satisfies
$$
\left\{\begin{aligned}
-&\Delta w=1, \; w>0&& \quad\mbox{ in }\Omega\setminus (K_1\cup L),\\
&w=0 &&\quad\mbox{ on }\partial(\Omega\setminus (K_1\cup L)).
\end{aligned}\right.
$$
Therefore, problem \eq{edecu} has a solution $u$. The uniqueness
follows from Lemma \ref{comp} while the asymptotic behavior of $u$
around $K_1$ is obtained in the same manner as in Lemma
\ref{comp3}. This ends the proof.
\end{proof}

Several times in this paper we shall use the following elementary
results that provide an equivalent integral condition to
\eq{vphioptim}.

\begin{lemma}\label{zero}
Let $N\geq 3$ and $\vphi:(0,\infty)\ri [0,\infty)$ be a continuous
function.
\begin{enumerate}
\item[(i)] $\displaystyle \int_0^1 r\vphi(r) dr<\infty$ if and
only if $\displaystyle\int_0^1 t^{1-N}\int_0^t
s^{N-1}\vphi(s)dsdt<\infty$; \item[(ii)]
$\displaystyle\int_1^\infty r\vphi(r) dr<\infty$ if and only if
$\displaystyle\int_1^\infty t^{1-N}\int_1^t
s^{N-1}\vphi(s)dsdt<\infty$; \item[(iii)]
$\displaystyle\int_0^\infty r\vphi(r) dr<\infty$ if and only if
$\displaystyle\int_0^\infty t^{1-N}\int_0^t
s^{N-1}\vphi(s)dsdt<\infty$;
\end{enumerate}
\end{lemma}
\begin{proof}
We prove only (i). The proof of (ii) is similar, while (iii)
follows from (i)-(ii).

Assume first that $\int_0^1 r\vphi(r) dr<\infty$. Integrating by
parts we have
$$
\begin{aligned}
\int_0^1 t^{1-N}\int_0^t
s^{N-1}\vphi(s)dsdt&=-\frac{1}{N-2}\int_0^1\Big(t^{2-N}\Big)'\int_0^t
s^{N-1}\vphi(s)dsdt\\
&=\frac{1}{N-2}\left(\int_0^1t\vphi(t)dt-\int_0^1
t^{N-1}\vphi(t)dt\right)\\
&\leq \frac{1}{N-2}\int_0^1t\vphi(t)dt<\infty.
\end{aligned}
$$
Conversely, for $0<\ep<1/2$ we have
$$
\begin{aligned}
\int_\ep^1 t^{1-N}\int_0^t
s^{N-1}\vphi(s)dsdt&=\frac{1}{N-2}\left(\int_\ep^1t\vphi(t)dt-\int_0^1
t^{N-1}\vphi(t)dt+\ep^{2-N}\int_0^\ep t^{N-1}\vphi(t)dt    \right)\\
&\geq \frac{1}{N-2}\left(\int_\ep^1t\vphi(t)dt-\int_\ep^1
t^{N-1}\vphi(t)dt\right)\\
&=\frac{1}{N-2}\int_\ep^1(1-t^{N-2})t\vphi(t)dt\\
&\geq \frac{1}{N-2}\left(1-\Big(\frac{1}{2}\Big)^{N-2} \right)
\int_\ep^{1/2}t\vphi(t)dt.
\end{aligned}
$$
Passing to the limit with $\ep\searrow 0$ we deduce $\int_0^1
t\vphi(t)dt<\infty$. This concludes the proof of Lemma \ref{zero}.
\end{proof}

\section{Proof of Theorem \ref{thoptim2}}

We start first with two nonexistence results that will help us to
prove the necessary part in Theorem \ref{thoptim2}.

\begin{prop}\label{thvv} Let $\vphi:(0,\infty)\ri [0,\infty)$ and
$f:(0,\infty)\ri(0,\infty)$ be continuous functions such that:
\begin{enumerate}
\item[(i)] $\liminf_{t\searrow 0} f(t)>0$;  \item[(ii)] $\vphi(r)$
is monotone for $r$ large; \item[(iii)] $\int_1^\infty
r\vphi(r)\,dr=\infty$;
\end{enumerate}
Then, for any compact set $K\subset \RR^N$ $(N\geq 3)$ there does
not exist a $C^2$ positive solution $u(x)$ of \eq{phi}.
\end{prop}
\begin{proof} It is easy to construct a $C^1$ function
$g:[0,\infty)\ri (0,\infty)$ such that $g<f$ in $(0,\infty)$ and
$g'$ is negative and nondecreasing. Therefore, we may assume
$f:[0,\infty)\ri (0,\infty)$ is of class $C^1$ and $f'$ is
negative and nondecreasing.

Suppose for contradiction that $u(x)$ is a $C^2$ positive solution
of \eq{phi}. By translation, we may assume that $0\in K$. Choose
$r_0>0$ such that
$$
K\subset B_{r_0/2}(0), \quad \vphi(r_0/2)>0, \quad\mbox{and}\quad
\vphi\mbox{ is monotone on }[r_0/2,\infty).
$$

Define $\psi:[r_0/2,\infty)\ri (0,\infty)$ by
\[
\psi(r)=\min_{r_0/2\leq \rho\leq r}\vphi(\rho)= \left\{
\begin{aligned}
\vphi(r)&& \quad\mbox{ if $\vphi$ is nonincreasing for }r\geq r_0/2,\\
\vphi(r_0/2)&& \quad\mbox{ if $\vphi$ is nondecreasing for }r\geq
r_0/2.
\end{aligned}\right.
\]
Then $\int_{r_0}^\infty r\psi(r) dr=\infty$. Also, since
$r_0/2\leq \delta_K(x)\leq |x|$ for all $x\in \RR^N\setminus
B_{r_0}(0)$, we have
\[
\vphi(\delta_K(x))\geq \psi(|x|)\quad\mbox{ for all } x\in
\RR^N\setminus B_{r_0}(0).
\]
Thus, the solution $u$ of \eq{phi} satisfies
\begin{equation}\label{psirad}
-\Delta u\geq \psi(|x|)f(u)\quad\mbox{ in }\RR^N\setminus
B_{r_0}(0).
\end{equation}

Averaging \eq{psirad} and using Jensen's inequality, we get
\begin{equation}\label{ave}
-\Big(\bar u''(r) + \frac{n-1}{r}\bar u'(r)\Big)\ge \psi(r)f(\bar
u(r)) \quad \text{for all }  r\ge r_0.
\end{equation}
Here $\bar u$ is the spherical average of $u$, that is
\begin{equation}\label{sfav}
\bar u(r)=\frac{1}{\sigma_N r^{N-1}}\int_{\partial B_r(0)}u(x)\,
d\sigma(x),
\end{equation}
where $\sigma$ denotes the surface area measure in $\RR^N$ and
$\sigma_N=\sigma(\partial B_1(0))$.

Making in \eq{ave} the change of variables $\bar u(r)=v(\rho)$,
$\rho=r^{2-N}$ we get
\[
-v''(\rho)\ge
\frac{1}{(N-2)^2}\rho^{2(N-1)/(2-N)}\psi(\rho^{1/(2-N)})f(v(\rho))
\quad\mbox{for all }0<\rho\leq \rho_0,
\]
where $\rho_0=r_0^{2-N}$.  Since $v$ is concave down and positive,
$v$ is bounded for $0<\rho\le \rho_0$. Hence $f(v(\rho)) \ge
(N-2)^2C$ for some positive constant $C$. Consequently
\[
-v''(\rho)\ge C\rho^{2(N-1)/(2-N)}\psi(\rho^{1/(2-N)}) \quad
  \text{for all }  0<\rho\le \rho_0.
\]
Integrating this inequality twice we get
\begin{align}
\infty & > \int_0^{\rho_0} v'(\rho)\,d\rho - \rho_0v'(\rho_0) \nonumber\\
       & \ge C\int_0^{\rho_0}\int_\rho^{\rho_0}
             \bar \rho^{2(N-1)/(2-N)}\psi(\bar \rho^{1/(2-N)})
             \,d\bar\rho\,d\rho \nonumber\\
       & = C\int_0^{\rho_0} \bar\rho^{1+ 2(N-1)/(2-N)}
             \psi(\bar \rho^{1/(2-N)})\,d\bar \rho \nonumber\\
       & = (N-2)C\int_{r_0}^\infty r\psi(r)\,dr=\infty. \nonumber
\end{align}
This contradiction completes the proof.
\end{proof}

\begin{prop}\label{thv} Let $\vphi:(0,\infty)\ri [0,\infty)$ and
$f:(0,\infty)\ri(0,\infty)$ be continuous functions such that
\begin{enumerate}
\item[(i)] $\liminf_{t\searrow 0} f(t)>0$; \item[(ii)] $\int_0^1
r\vphi(r)\,dr=\infty$.
\end{enumerate}
Then there does not exist a $C^2$ positive solution $u(x)$ of
\begin{equation}\label{or}
 -\Delta u\geq \vphi(\delta_{\o}(x))f(u)\quad \mbox{ in
}\{x\in\RR^N\setminus \oo:0<\delta_\o(x)<2r_0\}, \quad N\ge2,
\end{equation}
where $\o$ is a $C^2$ bounded domain in $\RR^N$ and $r_0>0$.
\end{prop}

For the proof of Proposition \ref{thv} we shall use the following
lemma concerning the geometry of a $C^2$ bounded domain. One can
prove it using the methods from \cite[page 96]{lang}.

\begin{lemma}\label{diffgeom} Let $\o$ be a $C^2$ bounded domain in $\RR^N$, $N\geq 2$, such
that $\RR^N\setminus \o$ is connected. Then, there exists $r_0>0$
such that
\begin{enumerate}
\item[(i)] $\o_r\colon\!\!\! =\{x\in\RR^N:{\rm dist}(x,\o)<r\}$ is
a $C^1$ domain for each $0<r\leq r_0$; \item[(ii)]for $0\leq r\leq
r_0$ the function $T(\cdot,r):\partial\o\ri \RR^N$ defined by
$T(\xi,r)=\xi+r\eta_\xi$, where $\eta_\xi$ is the outward unit
normal to $\partial\o$ at $\xi$, is a $C^1$ diffeomorphism from
$\partial\o$ onto $\partial\o_r$ (onto $\partial\o$ if $r=0$)
whose volume magnification factor (i.e., the absolute value of its
Jacobian determinant) $J(\cdot,r):\partial\o\ri (0,\infty)$ is
continuous on $\partial\o$ and $C^\infty$ with respect to $r$;
\item[(iii)] if $\eta_{T(\xi,r)}$ is the unit outward normal to
$\partial\o_r$ at $T(\xi,r)$ then $\eta_{T(\xi,r)}$ and $\eta_\xi$
are equal (but have different base points) for $\xi\in\partial\o$
and $0\leq r\leq r_0$.
\end{enumerate}
\end{lemma}

\begin{proof}[ Proof of Proposition \ref{thv}] Without loss of generality we can assume $\RR^N\setminus \o$ is connected.
Suppose for contradiction that $u(x)$ is a $C^2$ positive solution
of \eq{or}. By decreasing $r_0$ if necessary, the conclusion of
Lemma \ref{diffgeom} holds.

\begin{lemma}\label{lg}
The function
\[
g(r)=\int_{\partial\o_r}u(x)\,d\sigma(x),\quad 0<r\leq r_0,
\]
is continuously differentiable and there exists a positive
constant $C$ such that
$$
\left|g'(r)- \int_{\partial\o_r}\frac{\partial
u}{\partial\eta}\,d\sigma(x)\right|\leq Cg(r)\quad\mbox{ for all
}0<r\leq r_0,
$$
where $\eta$ is the outward unit normal to $\partial\o_r$.
\end{lemma}

\begin{proof}[Proof of Lemma \ref{lg}] By Lemma \ref{diffgeom} we
have
$$
g(r)=\int_{\partial\o}u(\xi+r\eta_\xi)J(\xi,r)\,d\sigma(\xi)\quad\mbox{
for all } 0<r\leq r_0,
$$
and thus
\begin{equation}\label{cinci}
\begin{aligned}
g'(r)&=\int_{\partial\o}\left[\frac{\partial}{\partial r}
\Big(u(\xi+r\eta_\xi)\Big)\right]J(\xi,r)\,d\sigma(\xi)+
\int_{\partial\o}u(\xi+r\eta_\xi)J_r(\xi,r)\,d\sigma(\xi)\\
&=\int_{\partial\o_r}\frac{\partial u}{\partial \eta}(x)
\,d\sigma(x)+
\int_{\partial\o_r}u(x)\frac{J_r(\xi,r)}{J(\xi,r)}\,d\sigma(x),
\end{aligned}
\end{equation}
for all $0<r\leq r_0$, where in the last integral
$\xi=x-r\eta_\xi\in\partial\o$. Since, by Lemma \ref{diffgeom},
$J(\xi,r)$ is positive and continuous for $\xi\in\partial\o$ and
$0\le r\le r_0$ and $J_r(\xi,r)$ is continuous there, we see that
Lemma \ref{lg} follows from \eq{cinci}.
\end{proof}

We now come back to the proof of Proposition \ref{thv}. For
$0<r\le r_0$ we have
\begin{equation}\label{sase}
\begin{aligned}
0&\leq \int_{\o_{r_0}\setminus\o_r}-\Delta u(x)\,dx
=\int_{\partial\o_r}\frac{\partial
u}{\partial\eta}\,d\sigma(x)-\int_{\partial\o_{r_0}}\frac{\partial
u}{\partial\eta}\,d\sigma(x)\\
&\leq g'(r)+Cg(r)+C
\end{aligned}
\end{equation}
for some positive constant $C$ by Lemma \ref{lg}. Hence
$$
\Big(e^{Cr}(g(r)+1)\Big)'\geq 0 \quad\mbox{ for all }0<r\leq r_0,
$$
and integrating this inequality over $[r,r_0]$ we obtain
\begin{equation}\label{ee2}
g(r)\leq e^{C(r_0-r)}(g(r_0)+1)-1\leq C_1\quad\mbox{ for all
}0<r\le r_0
\end{equation}
and for some $C_1>0$. Thus
$$
U(r)\colon\!\!\!=\frac{1}{|\partial
\o_r|}\int_{\partial\o_r}u(x)\,d\sigma(x)=\frac{g(r)}{|\partial
\o_r|}
$$
is bounded for $0<r\le r_0$. Consequently, by the assumption (i)
of $f$, it follows that
\begin{equation}\label{ee}
{|\partial \o_\rho|}f(U(\rho))\geq \ep>0\quad\mbox{ for all
}0<\rho\leq r_0.
\end{equation}
As in the proof of Proposition \ref{thvv}, we may assume that
$f:[0,\infty)\ri (0,\infty)$ is of class $C^1$ and $f'$ is
negative and nondecreasing. From \eq{or}, \eq{sase}-\eq{ee} and
Jensen's inequality we now obtain
$$
\begin{aligned}
g'(r)+C_2&\geq \int_{\o_{r_0}\setminus\o_r}-\Delta u\,dx\\
&\geq \int_r^{r_0}\vphi(\rho)\int_{\partial\o_\rho}f(u(x))\,d\sigma(x)\,d\rho\\
&\geq \int_r^{r_0}\vphi(\rho)|\partial\o_\rho|f(U(\rho))\,d\rho\\
&\geq\ep \int_r^{r_0}\vphi(\rho)\,d\rho\quad\mbox{ for all
}0<r\leq r_0.
\end{aligned}
$$
Integrating over $[r,r_0]$ in the above estimate we find
$$
\begin{aligned}
g(r_0)-g(r)+C_2r_0
&\geq \ep \int_r^{r_0}\int_s^{r_0}\vphi(\rho)\,d\rho\,ds\\
&= \ep \int_r^{r_0}(\rho-r) \vphi(\rho)\,d\rho\ri
\ep\int_0^{r_0}\rho\vphi(\rho)\,d\rho=\infty\quad\mbox{ as
}r\searrow 0,
\end{aligned}
$$
which contradicts $g>0$ and completes the proof.
\end{proof}

\noindent{\bf Proof of Theorem \ref{thoptim2}}
 The necessity of \eq{vphioptim} follows
from Propositions \ref{thvv} and \ref{thv}. To prove the
sufficiency part we shall separately analyse the cases
$K_2=\emptyset$ and $K_2\neq\emptyset$.

\subsection{Case $K_2=\emptyset$}

Assume first that $\RR^N\setminus K$ is connected and let
$0<\rho<R$ be such that $K\subset B_\rho(0)$. By Lemma \ref{comp3}
there exists
$$
u\in C^2(B_\rho(0)\setminus K)\cap C(\overline B_\rho(0)\setminus
int(K))
$$
such that
\begin{equation}\label{edecxy}
\left\{\begin{aligned}
-&\Delta u=\vphi(\delta_K(x))f(u), \; u>0&& \quad\mbox{ in }B_\rho(0)\setminus K,\\
&u=0 &&\quad\mbox{ on }\partial(B_\rho(0)\setminus K).
\end{aligned}\right.
\end{equation}
We next construct a solution $v$ of \eq{phi} in a neighborhood of
infinity. To this aim, let
$$
A(r):=\int_r^\infty t^{1-N}\int_R^t
s^{N-1}\vphi(s-\rho)dsdt\quad\mbox{ for all } r\geq R.
$$
Since $\int_R^\infty r\vphi(r-\rho)dr<\infty$, by Lemma \ref{zero}
we have that $A$ is well defined for all $r\geq R$. Also, it is
easy to check that
$$
-\Delta A(|x|)=\vphi(|x|-\rho)\quad\mbox{ in }\RR^N\setminus
B_R(0).$$

Since the mapping
$$
[0,\infty)\ni t\longmapsto \int_0^t \frac{1}{f(s)}ds\in [0,\infty)
$$ is bijective, we can define
$v:\RR^N\setminus B_R(0)\ri (0,\infty)$ implicitly as the unique
solution of
\begin{equation}\label{zee}
\int_0^{v(x)}\frac{1}{f(t)}dt =A(|x|)\quad\mbox{ for all
}x\in\RR^N\setminus B_R(0).
\end{equation}
Then, using the properties of $A$ we deduce that $v\in
C^2(\RR^N\setminus B_R(0))$, $v>0$ and $v(x)\rightarrow 0$ as
$|x|\rightarrow \infty$. Further from \eqref{zee} we obtain
$$
\nabla A(|x|)=\frac{1}{f(v)}\nabla v\quad\mbox{ in }\RR^N\setminus
B_R(0),
$$
and
$$
\vphi(|x|-\rho)=-\Delta A(|x|)=\frac{f'(v)}{f^2(v)}|\nabla
v|^2-\frac{1}{f(v)}\Delta v\quad\mbox{ in }\RR^N\setminus B_R(0).
$$
Since $f$ is decreasing, we have $f'\leq 0$ which implies
$$
-\Delta v\geq \vphi(|x|-\rho)f(v)\quad\mbox{ in }\RR^N\setminus
B_R(0). $$ Therefore, $v\in C^2(\RR^N\setminus B_R(0))$ satisfies
\begin{equation}\label{edecxx}
\left\{\begin{aligned}
-&\Delta v\geq \vphi(\delta_K(x))f(v), \; v>0&& \quad\mbox{ in }\RR^N\setminus B_R(0),\\
&v(x)\ri 0 &&\quad\mbox{ as }|x|\ri\infty.
\end{aligned}\right.
\end{equation}
Let now $0<\rho_0<\rho$ be such that $K\subset B_{\rho_0}(0)$ and
let $u$, $v$ be the solutions of \eq{edecxy} and \eq{edecxx}
respectively. Consider
$$
w:(B_{\rho_0}(0)\setminus int(K)) \cup (\RR^N\setminus B_R(0))\ri
[0,\infty), $$ defined by
$$
w(x)=u(x)\mbox{ if }x\in B_{\rho_0}(0)\setminus int(K), \quad
w(x)=v(x)\mbox{ if }x\in \RR^N\setminus B_{R}(0).
$$
Let $W$ be a positive $C^2$ extension of $w$ to $\RR^N\setminus
K$. We claim that there exists $M>0$ large enough such that
\begin{equation}\label{uphi}
U(x)=W(x)+M(1+|x|^2)^{(2-N)/2}, \quad x\in\RR^N\setminus int(K)
\end{equation}
satisfies \eq{phi}. Indeed, since $(1+|x|^2)^{(2-N)/2}$ is
superharmonic, this condition is already satisfied in
$B_{\rho_0}(0)\setminus K$ and $\RR^N\setminus B_R(0)$. In
$B_{R}(0)\setminus B_{\rho_0}(0)$ we use the fact that $-\Delta
(1+|x|^2)^{(2-N)/2}$ is positive and bounded away from zero.
Therefore we have constructed a solution $U\in C^2(\RR^N\setminus
K)\cap C(\RR^N\setminus int(K))$ of \eq{phi} that tends to zero at
infinity.

Assume now that $\RR^N\setminus K$ is not connected. We shall
construct a solution to \eq{phi} by considering each component of
$\RR^N\setminus K$. Note that since $K$ is compact,
$\RR^N\setminus K$ has only one unbounded component on which we
proceed as above. Since $\vphi$ satisfies \eq{vphioptim}, by Lemma
\ref{comp2}, on each bounded component of $\RR^N\setminus K$ we
construct a solution of $-\Delta u=\vphi(\delta_K(x))f( u)$ that
vanishes continuously on the boundary and has the behavior
described by \eq{bestim}.

\subsection{Case $K_2\neq\emptyset$}

Proceeding in the same manner as above (see \eq{uphi}) we can find
a function
\[
U\in C^2(\RR^N\setminus K_1)\cap C(\RR^N\setminus int(K_1))
\]
such that
\begin{equation}\label{u}
\left\{\begin{aligned}
-&\Delta U\geq \vphi(\delta_{K_1}(x))f(U), \; U>0&& \quad\mbox{ in }\RR^N\setminus K_1,\\
&U(x)\ri 0 &&\quad\mbox{ as }|x|\ri\infty.
\end{aligned}\right.
\end{equation}
We next construct a function $V\in C^2(\RR^N\setminus K_{2})\cap
C(\RR^N)$ such that
\begin{equation}\label{v}
\left\{\begin{aligned}
-&\Delta V\geq \vphi(\delta_{K_2}(x))f(V), \; V>0&& \quad\mbox{ in }\RR^N\setminus K_2,\\
&V(x)\ri 0 &&\quad\mbox{ as }|x|\ri\infty.
\end{aligned}\right.
\end{equation}
Using \eq{vphioptim} and Lemma \ref{zero}(iii), the function
$$
D(r):=\int_{r}^\infty t^{1-N}\int_0^t
s^{N-1}\vphi(s)dsdt\quad\mbox{ for all } r\geq 0,
$$
is well defined and $ -\Delta D(|x|)=\vphi(|x|)$ in
$\RR^N\setminus \{0\}$. We next define $v:\RR^N\ri (0,\infty)$ by
$$
\int_0^{v(x)}\frac{1}{f(t)}dt =D(|x|)\quad\mbox{ for all
}x\in\RR^N.
$$
Using the same arguments as in the previous case we have $v\in
C^2(\RR^N\setminus \{0\})\cap C(\RR^N)$ and
\begin{equation}\label{edecaaa}
\left\{\begin{aligned}
-&\Delta v\geq \vphi(|x|)f(v), \; v>0&& \quad\mbox{ in }\RR^N\setminus \{0\},\\
&v(x)\ri 0 &&\quad\mbox{ as }|x|\ri\infty.
\end{aligned}\right.
\end{equation}
Let now $V:\RR^N\ri (0,\infty)$ defined by
$$
V(x)=\sum_{a\in K_2}v(x-a).
$$
By \eq{edecaaa} we have $V\in C^2(\RR^N\setminus K_{2})\cap
C(\RR^N)$, $V(x)\ri0$ as $|x|\ri\infty$ and
$$
\begin{aligned}
-\Delta V(x)&=-\sum_{a\in K_2}\Delta v(x-a)
\geq \sum_{a\in K_2}\vphi(|x-a|)f(v(x-a))\\
&\geq \Big(\sum_{a\in K_2}\vphi(|x-a|)\Big)f(V(x))
\geq \vphi(\min_{a\in K_2}|x-a|)f(V(x))\\
&=\vphi(\delta_{K_2}(x))f(V(x))\quad\mbox{ for all }x\in
\RR^N\setminus K_2.
\end{aligned}
$$
Therefore, $V$ fulfills \eq{v}. Now
\begin{equation}\label{wphii}
W:=U+V:\RR^N\setminus K\ri \RR
\end{equation}
satisfies $W(x)\ri0$ as $|x|\ri\infty$ and
$$
\begin{aligned}
-\Delta W(x)&\geq \vphi(\delta_{K_1}(x))f(U)+\vphi(\delta_{K_2}(x))f(V)\\
&\geq (\vphi(\delta_{K_1}(x))+\vphi(\delta_{K_2}(x)))f(W)\\
&\geq \vphi\Big(\min\{\delta_{K_1}(x), \delta_{K_2}(x)\}\Big)f(W)\\
&=\vphi(\delta_{K}(x))f(W)\quad\mbox{ for all }x\in \RR^N\setminus
K.
\end{aligned}
$$
Thus, $W$ is a solution of \eq{phi} and the proof of Theorem
\ref{thoptim2} is now complete. \qed

\begin{remark} The approach in Theorem \ref{thoptim2} can be used to study
the inequality \eq{phi} in some cases where the compact set $K$
consists of infinitely many components all of them with $C^2$
boundary. For instance, it is easy to see that the same arguments
apply for compact sets $K$ of the form
$$
K=B_1(0)\cup\bigcup_{n\geq 1}\left\{x\in\RR^N:
1+\frac{1}{2n+1}<|x|<1+\frac{1}{2n}\right\}
$$
or
$$
K=\partial B_1(0)\cup  \bigcup_{n\geq 1}\partial B_{1+1/n}(0).
$$
\end{remark}

\smallskip

\section{Proof of Theorem \ref{thoptim3}}

\subsection{Case $K_2=\emptyset$}

We shall assume that $\RR^N\setminus K$ is connected as using the
arguments in the proof of Theorem \ref{thoptim2} on any bounded
component of $\RR^N\setminus K$ we can construct a solution of
$-\Delta u=\vphi(\delta_K(x))f( u)$ that vanishes continuously on
its boundary and has the behavior described by \eq{bestim}.

According to Lemma \ref{comp3}, for any $n\geq 1$ there exists a
unique $$u_n\in C^2(B_{R+n}(0)\setminus K)\cap C(\overline
B_{R+n}(0)\setminus int(K)) $$ such that
\begin{equation}\label{edecxyz}
\left\{\begin{aligned}
-&\Delta u_n=\vphi(\delta_K(x))f(u_n), \; u_n>0&& \quad\mbox{ in }B_{R+n}(0)\setminus K,\\
&u_n=0 &&\quad\mbox{ on }\partial(B_{R+n}(0)\setminus K).
\end{aligned}\right.
\end{equation}
We extend $u_n=0$ on $\RR^N\setminus B_{R+n}(0)$ and by Lemma
\ref{comp} we have that $\{u_n\}$ is a nondecreasing sequence of
functions and $u_n\leq U$ in $\RR^N\setminus K$. Let
\[
\tilde u(x)=\lim_{n\ri\infty}u_n(x) \quad\mbox{ for all }
x\in\RR^N\setminus int(K). \] By standard elliptic arguments, we
have $\tilde u\in C^2(\RR^N\setminus K)$ and
\[
-\Delta \tilde u=\vphi(\delta_K(x))f(\tilde u)\quad\mbox{ in
}\RR^N\setminus K.
\]
We next prove that $\tilde u$ vanishes continuously on $\partial
K$.

Let $u_1$ be the unique solution of \eq{edecxyz} with $n=1$ and
$\omega:=\{x\in\RR^N\setminus K: 0<\delta_K(x)<1\}$. Since both
$u_1$ and $\tilde u$ are continuous and positive on
$\partial\omega\setminus K$, one can find $M>1$ such that
$Mu_1\geq \tilde u$ on $\partial\omega\setminus K$. Now, using the
fact that the sequence $\{u_n\}$ is nondecreasing, this also
yields $Mu_1\geq u_n$ on $\partial\omega\setminus K$, for all
$n\geq 1$. The above inequality also holds true on $\partial K$
(since $u_1$ and $u_n$ are zero there). Therefore $Mu_1\geq u_n$
on $\partial\omega$ for all $n\geq 1$ which by the comparison
result in Lemma 2.1 (note that the function $Mu_1$ satisfies
\eq{phi} in $\omega$) gives
$$
Mu_1\geq u_n \quad\mbox{ in } \omega,
$$
for all $n\geq 1$. Passing to the limit with $n\ri\infty$ in the
above estimate, we obtain $Mu_1\geq \tilde u$ in $\omega$ and
since $u_1$ vanishes continuously on $\partial K$, so does $\tilde
u$.

The boundary behavior of $\tilde u$ near $K$ follows from the fact
that $u_1\leq \tilde u\leq Mu_1$ in $\omega$ and $u_1$ satisfies
\eq{besti}. From Lemma \ref{comp} we obtain that any solution $u$
of \eq{phi} satisfies $u\geq u_n$ in $\RR^N\setminus K$ which
implies $u\geq \tilde u$. Hence, $\tilde u$ is the minimal
solution of \eq{phi}.

\subsection{Case $K_2\neq\emptyset$}

Using, if necessary, a dilation argument, we can assume that
dist$(K_1,K_2)>2$ and the distance between any two distinct points
of $K_2$ is greater than 2. We fix $R>0$ large enough such that
$$
K_1\cup \bigcup_{a\in K_2}\overline B_1(a)\subset B_R(0).
$$
We now apply Lemma \ref{comp4} for $L=\bigcup_{a\in K_2}\overline
B_{1/n}(a)$ and $\Omega=B_{R+n}(a)$. Thus, there exists a unique
solution $u_n$ of
\begin{equation}\label{edecma}
\left\{\begin{aligned} -&\Delta u_n=\vphi(\delta_K(x))f(u_n), \;
u_n>0&& \quad\mbox{ in }B_{R+n}(0)\setminus
\Big(K_1\cup\bigcup_{a\in K_2}\overline
B_{1/n}(a)  \Big),\\
&u_n=0 &&\quad\mbox{ on }\partial B_{R+n}(0)\cup \partial K_1\cup
\bigcup_{a\in K_2}\partial B_{1/n}(a).
\end{aligned}\right.
\end{equation}
Extending $u_n=0$ outside of $\overline B_{R+n}(0)\setminus
\bigcup_{a\in K_2}\overline B_{1/n}(a)$, by Lemma \ref{comp} we
obtain
$$
0\leq u_1\leq u_2\leq \cdots \leq u_n\leq u_{n+1}\leq
\cdots\quad\mbox{ in }\RR^N\setminus K.
$$
By Lemma \ref{comp} we obtain $u_n\leq W$ in $\RR^N\setminus K$,
where $W$ is defined by \eq{wphii}. Thus, passing to the limit in
\eq{edecma} and by elliptic arguments, we obtain that $\tilde
u:=\lim_{n\ri\infty}u_n$ satisfies
\[
-\Delta \tilde u=\vphi(\delta_K(x))f(\tilde u)\quad\mbox{ in
}\RR^N\setminus K.
\]
The fact that $\tilde u$ is minimal, vanishes continuously on
$\partial K_1$ and has the behavior near $\partial K_1$ as
described by \eq{bestim} follows exactly in the same way as in the
proof of Theorem \ref{thoptim2}.

It remains to prove that $\tilde u$ can be continuously extended
at any point of $K_2$ and $\tilde u>0$ on $K_2$. To this aim, we
state and prove the following auxiliary results.

\begin{lemma}\label{ls1} Let $r>0$ and $x\in\RR^N\setminus\partial B_r(0)$, $N\geq 3$. Then
$$
\frac{1}{\sigma_N r^{N-1}}\int_{\partial
B_r(0)}\frac{1}{|x-y|^{N-2}}d\sigma(y)=\left\{
\begin{aligned}
&\frac{1}{|x|^{N-2}} &&\quad\mbox{ if }\,|x|>r,\\
&\frac{1}{r^{N-2}} &&\quad\mbox{ if }\,|x|<r.
\end{aligned}
\right.
$$
\end{lemma}
\begin{proof}[Proof of Lemma \ref{ls1}]
Suppose first $|x|>r$. Then $u(y)=|y-x|^{2-N}$ is harmonic in
$B_{r+\ep}(0)$, for $\ep>0$ small. By the mean value theorem we
have
$$
\frac{1}{\sigma_N r^{N-1}}\int_{\partial
B_r(0)}\frac{1}{|x-y|^{N-2}}d\sigma(y)=u(0)=\frac{1}{|x|^{N-2}}.
$$
Assume now $|x|<r$. Since
$$
v(x):=\frac{1}{\sigma_N r^{N-1}}\int_{\partial
B_r(0)}\frac{1}{|x-y|^{N-2}}d\sigma(y)
$$
is harmonic and radially symmetric, it follows that $v$ is
constant in $B_r(0)$. Thus $v(x)=v(0)=r^{2-N}$ for $x\in B_r(0)$.
\end{proof}

\begin{lemma}\label{ls2} Let $u$ be a $C^2$ positive solution of
$$
-\Delta u\geq 0\quad\mbox{ in }B_{2r_1}(0)\setminus\{0\}, \;N\geq
2.
$$
Then
$$
u(x)\geq m:=\min_{|y|=r_1}u(y)\quad\mbox{for all } x\in\overline
B_{r_1}(0)\setminus \{0\}.
$$
\end{lemma}
\begin{proof}[Proof of Lemma \ref{ls2}] For $0<r_0<r_1$ define
$v_{r_0}:\RR^N\setminus\{0\}\ri \RR$ by
$$
v_{r_0}(x)=\frac{m(\Phi(r_0)-\Phi(|x|))}{\Phi(r_0)-\Phi(r_1)},
$$
where
$$
\Phi(r)=\left\{
\begin{aligned}
&\log\frac{1}{r} &&\quad\mbox{ if }N=2,\\
&\frac{1}{r^{N-2}} &&\quad\mbox{ if }N\geq 3.
\end{aligned}
\right.
$$
Then $v_{r_0}$ is harmonic in $\RR^N\setminus\{0\}$ and
$v_{r_0}\leq u$ on $\partial B_{r_1}(0)\cup \partial B_{r_0}(0)$.
Thus, by the maximum principle, $v_{r_0}\leq u$ in $\overline
B_{r_1}(0)\setminus B_{r_0}(0)$. Fix $x\in \overline
B_{r_1}(0)\setminus\{0\}$. Then, for $0<r_0<|x|$ we have $u(x)\geq
v_{r_0}(x)\ri m$ as $r_0\searrow 0$. This concludes the proof.
\end{proof}

\begin{lemma}\label{ls3} Let $\vphi, f:(0,\infty)\ri [0,\infty)$ be
continuous functions such that $\int_0^1 r\vphi(r)dr<\infty$.
Suppose that $u$ is a $C^2$ positive bounded solution of $-\Delta
u=\vphi(|x|)f(u)$ in a punctured neighborhood of the origin in
$\RR^N$, $N\geq 3$. Then, for some $L>0$ we have $u(x)\ri L$ as
$x\ri 0$.
\end{lemma}
\begin{proof}[Proof of Lemma \ref{ls3}] By Lemma \ref{ls2} we can find $r_0>0$ small such
that $u$ is bounded away from zero in $\overline
B_{r_0}(0)\setminus\{0\}$. Hence, for some $M>0$ we have
\begin{equation}\label{mmn}
f(u(x))\leq M\quad\mbox{ in } \overline B_{r_0}(0)\setminus\{0\}.
\end{equation}
For $x\in\RR^N$ let
$$
I(x):=\frac{1}{\sigma_N}\int_{B_{r_0}(0)}\frac{\vphi(y)f(u(y))}{|x-y|^{N-2}}dy.
$$
Then,
$$
I(x)=\int_0^{r_0}F(x,r)dr, \quad\mbox{ where }\quad
F(x,r)=\frac{\vphi(r)}{\sigma_N}\int_{\partial
B_r(0)}\frac{f(u(y))}{|x-y|^{N-2}}d\sigma(y).
$$
Since, by \eq{mmn} and Lemma \ref{ls1} we have
\begin{enumerate}
\item[(i)] $F(x,r)\leq M r\vphi(r)$ for $x\in\RR^N$ and $0<r<r_0$;
\item[(ii)] $\int_0^{r_0} r\vphi(r) dr<\infty$; \item[(iii)]
$F(x,r)\ri F(0,r)$ as $x\ri 0$ pointwise for $0<r<r_0$,
\end{enumerate}
it follows that $I$ is bounded in $\RR^N$ and by the dominated
convergence theorem,
\begin{equation}\label{dct}
I(x)\ri I(0)\quad\mbox{ as }x\ri 0.
\end{equation}
Since $v:=u-\frac{1}{N-2}I$ is harmonic and bounded in
$B_{r_0}(0)\setminus\{0\}$, it is well known that $\lim_{x\ri
0}v(x)$ exists. Thus, by \eq{dct}, $\lim_{x\ri 0}u(x)$ exists and
is finite.
\end{proof}

Now, the fact that the minimal solution $\tilde u$ can be
continuously extended on $K_2$ and $\tilde u>0$ on $K_2$ follows
by applying Lemma \ref{ls3} for each point of $K_2$. This finishes
the proof of Theorem \ref{thoptim2}. \qed

\begin{remark} The existence of a positive ground state
solution in the exterior of a compact set is a particular feature
of the case $N\geq 3$. Such solutions do not exist in dimension
$N=2$. Indeed, suppose that $u$ is a $C^2$ positive solution of
$$
-\Delta u\geq 0\quad\mbox{ in }\RR^2\setminus K,\quad u(x)\ri
0\quad\mbox{ as }|x|\ri \infty,
$$
where $K\subset \RR^2$ is a compact set, not necessarily with
smooth boundary. Choose $r_0>0$ such that $K\subset B_{r_0}(0)$
and let $m=\min_{|x|=r_0}u(x)>0$. For each $r_1>r_0$ define
$$
v_{r_1}:\RR^2\setminus\{0\}\ri \RR,\quad v_{r_1}(x)=\frac{m(\log
r_1-\log |x|)}{\log r_1-\log r_0}.
$$
Then
$$
v_{r_1}\mbox{ is harmonic in } \RR^2\setminus\{0\},\quad v_{r_1}=m
\mbox{ on }\partial B_{r_0}(0), \quad v_{r_1}=0 \mbox{ on
}\partial B_{r_1}(0).
$$
Let $w_{r_1}(x)=u(x)-v_{r_1}(x)$, $x\in\RR^N\setminus B_{r_0}(0)$.
Thus,
$$
-\Delta w_{r_1}=-\Delta u\geq 0\quad\mbox{ in }\overline
B_{r_1}(0) \setminus B_{r_0}(0),\quad w_{r_1}\geq 0\quad\mbox{ on
}\partial B_{r_1}(0)\cup \partial B_{r_0}(0).
$$
By the maximum principle it follows that $w_{r_1}\geq 0$ in $\overline
B_{r_1}(0)\setminus B_{r_0}(0)$, that is $u(x)\geq v_{r_1}(x)$ in
$\overline B_{r_1}(0)\setminus B_{r_0}(0)$.

Let now $x\in\RR^2\setminus\ \overline B_{r_0}(0)$ be fixed. Then,
for $r_1>|x|$ we have
$$
u(x)\geq v_{r_1}(x)\ri m\quad\mbox{ as }r_1\ri\infty,
$$
so $u(x)\geq m$ in $\RR^2\setminus \overline B_{r_0}(0)$, which
contradicts $u(x)\ri 0$ as $|x|\ri\infty$.
\end{remark}

\medskip

\section{Proof of Theorem \ref{k2}}
Assume first that \eq{phi} has a $C^2$ positive solution $u$. From
Proposition \ref{thvv} it follows that $\int_1^\infty
r\vphi(r)dr<\infty$. By translation one may assume that $0\in K$
and fix $r_0>0$ such that $\delta_K(x)=|x|$ for $0<|x|<r_0$. Let
now $u_*$ be the image of $u$ through the Kelvin transform, that
is,
\neweq{kelvin}
u_*(x)=|x|^{2-N} u\left(\frac{x}{|x|^2}  \right), \quad
x\in\RR^N\setminus B_{1/r_0}(0).
\endeq
Then $u_*$ satisfies
$$
\begin{aligned}
-\Delta u_*&\geq |x|^{-2-N} \vphi\left(\frac{1}{|x|}
\right)f\left( u\left(\frac{x}{|x|^2}\right)\right)\\
&=|x|^{-2-N} \vphi\left(\frac{1}{|x|} \right)f(|x|^{N-2}u_*(x))
\quad\mbox{ in }\RR^N\setminus B_{1/r_0}(0).
\end{aligned}
$$
By taking the spherical average of $u$ and then using the change
of variable $\rho=r^{2-N}$ as in the proof of Proposition
\ref{thvv} (note that here we do not need $\vphi(r)$ to be
monotone for small values of $r>0$) we deduce
\[
\int_1^\infty
t^{-1-N}\vphi\left(\frac{1}{t}\right)f(at^{N-2})dt<\infty.
\]
Now with the change of variable $r=t^{-1}$, $0<r\leq 1$ we derive
the condition \eq{eqp2}.

Conversely, assume now that \eq{eqp1}-\eq{eqp2} hold and let us
construct a solution to \eq{phi} in the case  $K=\{0\}$. This will
follow from lemma below.

\begin{lemma}\label{vab} Let $a>0$ be such that \eq{eqp1} and \eq{eqp2} hold. Then for all $b>0$ there
exists a radially symmetric function $v_{a,b}\in
C^2(\RR^N\setminus\{0\})$ such that
$$
-\Delta v_{a,b}\geq \vphi(|x|)f(v_{a,b})\quad\mbox{ in
}\RR^N\setminus\{0\},
$$
and
$$
\lim_{|x|\ri 0}|x|^{N-2}v_{a,b}(x)=a\,, \quad
\lim_{|x|\ri\infty}v_{a,b}(x)=b.
$$
\end{lemma}
\begin{proof}
Let $u_0(r)=ar^{2-N}+b$, $r>0$ and for all $n\geq 1$ define
inductively the sequence
\begin{equation}\label{seqdef}
u_n(r)=u_0+\int_r^\infty t^{1-N}\int_0^t s^{N-1} \vphi(s)
f(u_{n-1}(s))dsdt\,, \quad r>0.
\end{equation}
Remark first that $u_n$ is well defined since $u_{n-1}\geq u_0$
and by Lemma \ref{zero}  we have
$$
\begin{aligned}
\int_r^\infty & t^{1-N}\int_0^t s^{N-1} \vphi(s) f(u_{n-1}(s))dsdt \leq \int_0^\infty t^{1-N}\int_0^t s^{N-1} \vphi(s) f(u_{0}(s))dsdt\\
\leq & \int_0^\infty r \vphi(r) f(u_{0}(r))dr \leq \int_0^1
r^{1-N}\vphi(r) f(ar^{2-N})dr+ f(b)\int_1^\infty r\vphi(r)dr
<\infty.
\end{aligned}
$$
A straightforward induction argument yields
\begin{equation}\label{seqdef1}
u_1\geq u_{2n-1}\geq u_{2n+1}\geq u_{2n}\geq u_{2n-2}\geq u_0\,,
\end{equation}
for all $n\geq 1$. Thus, there exists $u(r):=\lim_{n\ri
\infty}u_{2n}(r)$ and $v(r):=\lim_{n\ri \infty}u_{2n-1}(r)$,
$r>0$. Passing to the limit in \eq{seqdef} and \eq{seqdef1} we
find
\begin{equation}\label{seqdef2}
\left\{
\begin{aligned}
u(r)&=u_0+\int_r^\infty t^{1-N}\int_0^t s^{N-1} \vphi(s) f(v(s))dsdt\,, \quad r>0,\\
v(r)&=u_0+\int_r^\infty t^{1-N}\int_0^t s^{N-1} \vphi(s)
f(u(s))dsdt\,,\quad r>0,
\end{aligned}
\right.
\end{equation}
and $v\geq u$. Thus $V(x)=v(|x|)$ satisfies
$$
-\Delta V(x)=\vphi(|x|)f(u(|x|))\geq \vphi(|x|)f(V(x))\quad\mbox{
in }\RR^N\setminus\{0\}.
$$
Since the integrals in \eq{seqdef2} are finite, it is easy to
check that
$$
\lim_{|x|\ri 0}|x|^{N-2}V(x)=a\,, \quad \lim_{|x|\ri\infty}V(x)=b.
$$
Therefore, $v_{a,b}\equiv V$ satisfies the requirements in Lemma
\ref{vab}.
\end{proof}
If $K$ is a finite set of points and $V$ is any solution of
$-\Delta V\geq \vphi(|x|) f(V)$ in $\RR^N\setminus\{0\}$ then
$U(x):=\sum_{y\in K}V(x-y)$ is a solution of \eq{phi}.

Under the conditions \eq{eqp1}-\eq{eqp2}, the existence of the
minimal solution $\tilde u$ of \eq{phi} is obtained with the same
proof as in Theorem \ref{thoptim2}. Note that $\tilde u$ is
obtained as a pointwise limit of the sequence $\{u_n\}$ where
$u_n$ satisfies \eq{edecma} in which $K_1=\emptyset$ and $K_2=K$.
It remains to prove that $\tilde u$ can be continuously extended
to a positive continuous function in $\RR^N$ if and only if
$\int_0^1 r\vphi(r)dr<\infty$.

Assume first that the minimal solution $\tilde u$ of \eq{phi} is
bounded. Using a translation argument, one can also assume that
$0\in K$. Fix $r_0>0$ such that $\delta_K(x)=|x|$ for all $x\in
\overline B_{r_0}(0)$. Then averaging \eq{phi} we obtain
\begin{equation}\label{avgg}
-(r^{N-1}\bar u'(r))'\geq  cr^{N-1}\vphi(r)\quad\mbox{for all }
0<r\leq r_0,
\end{equation}
where $c>0$. Hence $r^{N-1}\bar u'(r)$ is decreasing and its limit
as $r\searrow 0$ must be zero for otherwise $\bar u-$and hence
$u-$would be unbounded for small $r>0$. Thus integrating \eq{avgg}
twice we obtain
$$
\infty>\Big(\limsup_{r\searrow 0} \bar u(r)\Big)-\bar u(r_0)\geq
c\int_0^{r_0} t^{1-N}\int_0^t s^{N-1}\vphi(s)dsdt,
$$
which by Lemma \ref{zero}(ii) yields $\int_0^1
r\vphi(r)dr<\infty$.

Assume now that $\int_0^1 r\vphi(r)dr<\infty$. The conclusion will
follow by Lemma \ref{ls3} once we prove that $\tilde u$ is bounded
around each point of $K$. Again by translation and a scaling
argument we may assume that $0\in K$ and $\delta_K(x)=|x|$ for all
$x\in B_{1}(0)$. Set
$$
v(x):=M\int_{|x|}^{2} t^{1-N}\int_0^t s^{N-1}\vphi(s)dsdt,\quad
\mbox{ for all } x\in B_{2}(0).
$$
By Lemma \ref{zero}(i), $v$ is bounded and positive in $B_{2}(0)$
and
\begin{equation}\label{vvx}
-\Delta v(x)=M\vphi(|x|)=M\vphi(\delta_K(x))\quad\mbox{ in
}B_1(0)\setminus\{0\}.
\end{equation}
Therefore, we can take $M>1$ large enough such that
\begin{equation}\label{vvxx}
-\Delta v(x)\geq \vphi(\delta_K(x))f(v(x))\quad\mbox{ in
}B_1(0)\setminus\{0\}\quad\mbox{ and }\quad v\geq \tilde
u\quad\mbox{ on }\partial B_1(0).
\end{equation}
Let $u_n$ be the solution of \eq{edecma} with $K_1=\emptyset$ and
$K_2=K$. Recall that $\{u_n\}$ converges pointwise to $\tilde u$.
Since $\tilde u\geq u_n$ in $\RR^N\setminus K$, from \eq{vvxx} we
have $v\geq u_n$ on $\partial B_1(0)$. According to the definition
of $u_n$, this inequality also holds true on $\partial
B_{1/n}(0)$. Now, by \eq{vvxx} and Lemma \ref{comp} it follows
that $v\geq u_n$ in $B_1(0)\setminus B_{1/n}(0)$. Passing to the
limit with $n\ri\infty$ it follows that $v\geq \tilde u$ in
$B_1(0)\setminus\{0\}$ and so, $\tilde u$ is bounded around zero.
In a similar way we derive that $\tilde u$ is bounded around
every point of $K$. By Lemma \ref{ls3} we now obtain that $\tilde
u$ can be continuously extended on $K$. This finishes the proof of
Theorem \ref{k2}.

\qed

\section{Proof of Theorem \ref{thdeg2} }

We shall divide our arguments into three steps.

\smallskip

\noindent{\it Step 1:} There exists a minimal solution
$\xi:\RR^N\setminus\{0\}\ri (0,\infty)$ which in addition
satisfies
\begin{equation}\label{limx}
\lim_{|x|\ri 0}|x|^{N-2}\xi(x)=0\quad\mbox{ and }\quad
\lim_{|x|\ri \infty}\xi(x)=0.
\endeq
Indeed, by Lemma \ref{crt} there exists a unique function $\xi_n$
such that
\begin{equation}\label{qqwq}
\left\{\begin{aligned}
-&\Delta \xi_n=\vphi(|x|)f(\xi_n),\; \xi_n>0 && \quad\mbox{ in } B_n(0)\setminus\overline B_{1/n}(0),\\
&\xi_n=0&&\quad\mbox{ on } \partial B_n(0)\cup\partial B_{1/n}(0).
\end{aligned}\right.
\end{equation}
By uniqueness, it also follows that $\xi_n$ is radially symmetric.
We next extend $\xi_n=0$ outside $B_n(0)\setminus \overline
B_{1/n}(0)$. Now, by Lemma \ref{comp} we have that $\{\xi_n\}$ is
nondecreasing. For any $\ep>0$, let $v_{\ep}$ be the function
constructed in Lemma \ref{vab} for $a=\ep$ and $b=\ep$. Then,
again by Lemma \ref{comp} we have $\xi_n\leq v_{\ep}$ in
$\overline B_n(0)\setminus B_{1/n}(0)$.

Hence, there exists $\xi(x):=\lim_{n\ri\infty} \xi_{n}(x)$,
$x\in\RR^N\setminus\{0\}$ and $\xi \leq v_{\ep}$. Also $\xi$ is
radially symmetric and by standard elliptic arguments it follows
that $\xi$ is a solution of \eq{rad}. From $\xi\leq v_{\ep}$ it
follows that $\lim_{|x|\ri 0}|x|^{N-2}\xi(x)\leq \ep$ and
$\lim_{|x|\ri \infty}\xi(x)\leq \ep$. Now, since $\ep>0$ was
arbitrarily chosen, we deduce that  $\xi$ satisfies \eq{limx}.

Finally, if $u$ is an arbitrary solution of \eq{rad}, by Lemma
\ref{comp} we deduce $\xi_n\leq u$  in  $B_n(0)\setminus \overline
B_{1/n}(0)$ which next implies that $\xi\leq u$ in
$\RR^N\setminus\{0\}$. Therefore $\xi$ is the minimal solution of
\eq{rad}.

\smallskip

\noindent{\it Step 2:} Proof of (i).

Fix $a,b\geq 0$. We shall construct a radially symmetric solution
of \eq{rad} that satisfies \eq{lim} with the aid of the minimal
solution $\xi$ constructed at Step 1. By virtue of Lemma
\ref{crt}, for any $n\geq 2$ there exists a unique function
\[
u_n\in C^2(B_n(0)\setminus\overline B_{1/n}(0))\cap C(\overline
B_n(0)\setminus B_{1/n}(0)) \] such that
\begin{equation}\label{zeroa}
\left\{\begin{aligned}
-&\Delta u_n=|x|^{\alpha}u_n^{-p},\; u_n>0 && \quad\mbox{ in } B_n(0)\setminus\overline B_{1/n}(0),\\
&u_n=a|x|^{2-N}+b+\xi(x)&&\quad\mbox{ on } \partial
B_n(0)\cup\partial B_{1/n}(0).
\end{aligned}\right.
\end{equation}
Since $\xi$ is radially symmetric, so is $u_n$. Furthermore,
$a|x|^{2-N}+b$ is a sub-solution while $a|x|^{2-N}+b+\xi(x)$ is a
super-solution of \eqref{zeroa}. Thus, in view of Lemma
\ref{comp}, we obtain
\begin{equation}\label{eqvv}
a|x|^{2-N}+b\leq u_n(x)\leq a|x|^{2-N}+b+\xi(x) \quad \mbox{ in
}B_n(0)\setminus B_{1/n}(0).
\end{equation}
As before we extend $u_n=0$ outside $B_n(0)\setminus\overline
B_{1/n}(0)$. By standard elliptic regularity and a diagonal
process, up to a subsequence there exists
$$
u_{a,b}(x):=\lim_{n\rightarrow \infty}u_n(x),\quad
x\in\RR^N\setminus\{0\}$$ and $u_{a,b}$ is a solution of problem
\eq{rad}. Furthermore, from \eqref{eqvv} we deduce that $u_{a,b}$
satisfies
\begin{equation}\label{aeq}
a|x|^{2-N}+b\leq u_{a,b}(x)\leq a|x|^{2-N}+b+\xi(x)\quad\mbox{ in
} \RR^N\setminus\{0\}.
\end{equation} Now, \eqref{limx} and \eq{aeq} imply \eq{lim}.

\smallskip

\noindent{\it Step 3:}  Proof of (ii).

Let $u$ be a solution of \eq{rad}. By Lemma 2.4 in \cite{gmt} (see
also Brezis and Lions \cite{brezis}) we have $u\in
L^1_{loc}(\RR^N)$ so there exists $a\geq 0$ such that
$$
\Delta u+\vphi(|x|)u^{-p}+a\delta(0)=0\quad\mbox{ in
}\mathcal{D}'(\RR^N),
$$
where $\delta(0)$ denotes the Dirac mass concentrated at zero.
Now, by the representation formula in \cite[Theorem
2.4]{mitidieri} we have
$$
u(x)=a|x|^{2-N}+b+C(N)\int_{\RR^N}\frac{\vphi(|y|)f(u(y))}{|x-y|^{N-2}}
dy\quad\mbox{ in }\RR^N\setminus\{0\}.
$$
Since $\xi$ is also a solution of \eq{rad} that satisfies
\eq{limx} we have
$$
\xi(x)=C(N)\int_{\RR^N}\frac{\vphi(|y|)f(\xi(y))}{|x-y|^{N-2}}
dy\quad\mbox{ in }\RR^N\setminus\{0\}.
$$
Using now the monotonicity of $f$ we deduce
$$
a|x|^{2-N}+b\leq u\leq a|x|^{2-N}+b+\xi\quad\mbox{ in
}\RR^N\setminus\{0\}.
$$
This implies $\lim_{|x|\ri 0}|x|^{N-2} u(x)=a$ and $\lim_{|x|\ri
\infty}u(x)=b$.

Let now $u_{a,b}$ be the solution of \eq{rad} that satisfies
\eq{lim}. We claim that $u\equiv u_{a,b}$. To this end, for
$\varepsilon>0$ define
$$
u_\varepsilon(x):=u(x)+\varepsilon (|x|^{2-N}+1),\quad
x\in\RR^N\setminus\{0\}. $$ Then, we can find $R=R(\ep)>0$ such
that $u_{\ep}\geq u_{a,b}$ if $|x|>R$ or $0<|x|<1/R$. By means of
Lemma \ref{comp} the same inequality is true in $B_R(0)\setminus
B_{1/R}(0)$, so $u_{\ep}\geq u_{a,b}$ in $\RR^N\setminus \{0\}$.
Passing now to the limit with $\ep\ri 0$ it follows that $u\geq
u_{a,b}$ in $\RR^N\setminus\{0\}$. In the same way we obtain the
reverse inequality so $u\equiv u_{a,b}$. This finishes the proof
of our Theorem \ref{thdeg2}. \qed

\end{document}